\documentclass[12pt]{amsart}
\usepackage{graphics}
\usepackage{amsfonts,amssymb,color}
\usepackage[mathscr]{eucal}
\usepackage{amsmath, amsthm}
\usepackage{mathrsfs}
\usepackage{amsbsy}
\usepackage{dsfont}
\usepackage{bbm}
\usepackage{wasysym}
\usepackage{stmaryrd}
\usepackage{url}
\usepackage{fourier}
\usepackage[T1]{fontenc}

\input xypic
\xyoption{all}

\makeindex
\makeglossary

\begin{document}
\baselineskip = 16pt

\newcommand \ZZ {{\mathbb Z}}
\newcommand \NN {{\mathbb N}}
\newcommand \QQ {{\mathbb Q}}
\newcommand \RR {{\mathbb R}}
\newcommand \CC {{\mathbb C}}
\newcommand \PR {{\mathbb P}}
\newcommand \AF {{\mathbb A}}
\newcommand \GG {{\mathbb G}}
\newcommand  \bbD{{\mathcal D}}
\newcommand \bcA {{\mathscr A}}
\newcommand \bcC {{\mathscr C}}
\newcommand \bcD {{\mathscr D}}
\newcommand \bcF {{\mathscr F}}
\newcommand \bcG {{\mathscr G}}
\newcommand \bcH {{\mathscr H}}
\newcommand \bcM {{\mathscr M}}
\newcommand \bcJ {{\mathscr J}}
\newcommand \bcL {{\mathscr L}}
\newcommand \bcO {{\mathscr O}}
\newcommand \bcP {{\mathscr P}}
\newcommand \bcQ {{\mathscr Q}}
\newcommand \bcR {{\mathscr R}}
\newcommand \bcS {{\mathscr S}}
\newcommand \bcV {{\mathscr V}}
\newcommand \bcW {{\mathscr W}}
\newcommand \bcX {{\mathscr X}}
\newcommand \bcY {{\mathscr Y}}
\newcommand \bcZ {{\mathscr Z}}
\newcommand \goa {{\mathfrak a}}
\newcommand \gob {{\mathfrak b}}
\newcommand \goc {{\mathfrak c}}
\newcommand \gom {{\mathfrak m}}
\newcommand \gon {{\mathfrak n}}
\newcommand \gop {{\mathfrak p}}
\newcommand \goq {{\mathfrak q}}
\newcommand \goQ {{\mathfrak Q}}
\newcommand \goP {{\mathfrak P}}
\newcommand \goM {{\mathfrak M}}
\newcommand \goN {{\mathfrak N}}
\newcommand \uno {{\mathbbm 1}}
\newcommand \Le {{\mathbbm L}}
\newcommand \Spec {{\rm {Spec}}}
\newcommand \Gr {{\rm {Gr}}}
\newcommand \Pic {{\rm {Pic}}}
\newcommand \Jac {{{J}}}
\newcommand \Alb {{\rm {Alb}}}
\newcommand \Corr {{Corr}}
\newcommand \Chow {{\mathscr C}}
\newcommand \Sym {{\rm {Sym}}}
\newcommand \Prym {{\rm {Prym}}}
\newcommand \cha {{\rm {char}}}
\newcommand \eff {{\rm {eff}}}
\newcommand \tr {{\rm {tr}}}
\newcommand \Tr {{\rm {Tr}}}
\newcommand \pr {{\rm {pr}}}
\newcommand \ev {{\it {ev}}}
\newcommand \cl {{\rm {cl}}}
\newcommand \interior {{\rm {Int}}}
\newcommand \sep {{\rm {sep}}}
\newcommand \td {{\rm {tdeg}}}
\newcommand \alg {{\rm {alg}}}
\newcommand \im {{\rm im}}
\newcommand \gr {{\rm {gr}}}
\newcommand \op {{\rm op}}
\newcommand \Hom {{\rm Hom}}
\newcommand \Hilb {{\rm Hilb}}
\newcommand \Sch {{\mathscr S\! }{\it ch}}
\newcommand \cHilb {{\mathscr H\! }{\it ilb}}
\newcommand \cHom {{\mathscr H\! }{\it om}}
\newcommand \colim {{{\rm colim}\, }} 
\newcommand \End {{\rm {End}}}
\newcommand \coker {{\rm {coker}}}
\newcommand \id {{\rm {id}}}
\newcommand \van {{\rm {van}}}
\newcommand \spc {{\rm {sp}}}
\newcommand \Ob {{\rm Ob}}
\newcommand \Aut {{\rm Aut}}
\newcommand \cor {{\rm {cor}}}
\newcommand \Cor {{\it {Corr}}}
\newcommand \res {{\rm {res}}}
\newcommand \red {{\rm{red}}}
\newcommand \Gal {{\rm {Gal}}}
\newcommand \PGL {{\rm {PGL}}}
\newcommand \Bl {{\rm {Bl}}}
\newcommand \Sing {{\rm {Sing}}}
\newcommand \spn {{\rm {span}}}
\newcommand \Nm {{\rm {Nm}}}
\newcommand \inv {{\rm {inv}}}
\newcommand \codim {{\rm {codim}}}
\newcommand \Div{{\rm{Div}}}
\newcommand \sg {{\Sigma }}
\newcommand \DM {{\sf DM}}
\newcommand \Gm {{{\mathbb G}_{\rm m}}}
\newcommand \tame {\rm {tame }}
\newcommand \znak {{\natural }}
\newcommand \lra {\longrightarrow}
\newcommand \hra {\hookrightarrow}
\newcommand \rra {\rightrightarrows}
\newcommand \ord {{\rm {ord}}}
\newcommand \Rat {{\mathscr Rat}}
\newcommand \rd {{\rm {red}}}
\newcommand \bSpec {{\bf {Spec}}}
\newcommand \Proj {{\rm {Proj}}}
\newcommand \pdiv {{\rm {div}}}
\newcommand \CH {{\it {CH}}}
\newcommand \wt {\widetilde }
\newcommand \ac {\acute }
\newcommand \ch {\check }
\newcommand \ol {\overline }
\newcommand \Th {\Theta}
\newcommand \cAb {{\mathscr A\! }{\it b}}

\newenvironment{pf}{\par\noindent{\em Proof}.}{\hfill\framebox(6,6)
\par\medskip}

\newtheorem{theorem}[subsection]{Theorem}
\newtheorem{conjecture}[subsection]{Conjecture}
\newtheorem{proposition}[subsection]{Proposition}
\newtheorem{lemma}[subsection]{Lemma}
\newtheorem{remark}[subsection]{Remark}
\newtheorem{remarks}[subsection]{Remarks}
\newtheorem{definition}[subsection]{Definition}
\newtheorem{corollary}[subsection]{Corollary}
\newtheorem{example}[subsection]{Example}
\newtheorem{examples}[subsection]{examples}

\title{A note on Mumford-Roitman argument on Chow schemes}
\author{ Kalyan Banerjee}
\address{Indian Statistical Institute, Bangalore Center, Bangalore 560059}
\email{kalyanb$_{-}$vs@isibang.ac.in}

\footnotetext{Mathematics Classification Number: 14L40, 14L10}
\begin{abstract}
In this note we are going to understand two questions. One is the fiber of the natural map from a projective algebraic group $G$ to $G/\Gamma$, where $\Gamma$ denotes the $\Gamma$-equivalence on $G$. The other one is to define a natural map from Hilbert scheme of the generic fiber of a fibration $X\to S$ to the Chow group of relative zero cycles on $X\to S$ and to understand the fibers of this map.
\end{abstract}
\maketitle
\section{Introduction}
In the breakthrough paper \cite{M}, Mumford had sketched an outline of the fact that the fibers of the natural map from the symmetric powers of a smooth projective variety $X$ to the Chow group of $X$ are countable unions of Zariski closed subsets inside the symmetric powers of $X$. In the paper by \cite{R} Roitman has proven that the fibers are indeed countable union of Zariski closed subsets inside the symmetric powers of the smooth projective variety $X$. That is the departing point of this article. We ask the same question but for $\Gamma$-equivalence on projective algebraic groups. Here $\Gamma$ is a smooth projective curve. Two points $g,h$ on $G$ are said to be $\Gamma$-equivalent, if there exists two points $0,\infty$ on $\Gamma$, and a rational map $f$ from $\Gamma$ to $G$ such that

$$f(0)=g,\quad f(\infty)=h\;.$$

Now we consider the natural map $\theta$ from $G$ to $G/\Gamma$, where $\Gamma$ denotes the $\Gamma$-equivalence relation and ask what is the kernel of $\theta$ or the fiber of $\theta$ over $e$, the identity element of $G$. So our main theorem of this article is as follows.

\textit{Let $G$ be a projective algebraic group over an uncountable, algebraically closed ground field $k$. Let $\theta$ denote the natural map from $G$ to $G/\Gamma$. Then $\theta^{-1}([e])$ is a countable union of translates of an abelian variety $A_0$ of $G$.}

To prove that we mainly used the Roitman's technique to stratify $\theta^{-1}([e])$, in terms of quasi-projective schemes and show that the Zariski closure of each of them is again in $\theta^{-1}([e])$, obtaining that $\theta^{-1}([e])$ is a countable union of Zariski closed subsets of $G$. Then the uncountability of the ground field is coming into the picture, giving us the fact that one of these Zariski closed subsets is actually an abelian variety and we obtain the others as the translates of the abelian variety.
Over complex numbers the picture is much more interesting, because we can use the analytic structure of $G$, that it is a complex compact manifold and $\theta^{-1}([e])$ is a locally compact hausdorff topological subgroup of it. Hence it is a Baire subspace of $G$ giving us that only finite number of translates give us $\theta^{-1}([e])$. This is our final theorem.

\textit{Let $G$ be a projective algebraic group over $\CC$. Consider the natural map $\theta$ from $G$ to $G/\Gamma$. Then $\theta^{-1}([e])$ is a finite union of translates of an abelian subvariety $A_0$ of $G$. Hence $\theta^{-1}([e])$ is an algebraic subgroup of $G$.}

The next section is devoted to relative zero cycles which was first introduced by Suslin and Voevodsky in \cite{SV}. We define the Chow group of relative zero cycles and produce a natural map from the Hilbert scheme of length $d$ zero dimensional subschemes on the generic fiber of a fibration $X\to S$, ($X,S$ smooth projective) to the Chow group of relative zero cycles on $X\to S$. We prove that the fibers of this map is countable union of Zariski closed subschemes in the Hilbert scheme. Here we use the techniques coming from \cite{R} to prove this result, also a sketch of this proof was given by Mumford in \cite{M}.

{\small \textbf{Acknowledgements:} The author would like to thank the ISF-UGC grant for funding this project and is grateful to the hospitality of Indian Statistical Institute, Bangalore Center for hosting this project. The author also thanks the anonymous referee for pointing out an inaccuracy about $R$-equivalence in the earlier version of the paper. Finally the author is grateful to Vladimir Guletskii for telling the problem about generalization of the Mumford-Roitman argument for the case of relative cycles, to the author.}

\section{Preliminaries}
Let $G$ be a projective algebraic group over a ground field $k$. Two $k$-points on $G$ are said to be $\Gamma$-equivalent if there exists a chain of rational maps from $\Gamma$ connecting them. Precisely, let $a,b$ be two $k$-points on $G$. They are said to be $\Gamma$-equivalent if there exists rational morphisms $f:\Gamma\to G$ , $0,\infty$ in $\Gamma$  such that
$$f(0)=a,\quad f(\infty)=b\;.$$

 Let $\Gamma(e)$ be the  class of the identity $e$ of $G$ under the above relation, that is collection of all $k$-points of $G$, $\Gamma$-equivalent to $e$. Then $\Gamma(e)$ is a subgroup of $G$ and $G/\Gamma(e)$ is a group.

\subsection{Mumford-Roitman techniques}
Let us consider the following map $\theta:G\to G/\Gamma(e)$, define by
$$\theta(g)=[g]$$
where $[g]$ denotes the  class of $g$ in $G/\Gamma(e)$. Since the addition law in $G/\Gamma(e)$ is defined to be
$$[g]+[h]=[g+h]$$
we get that $\theta$ is a homomorphism of groups. We are interested to understand what is the kernel of $\theta$ or $\theta^{-1}([e])$. There was a similar such question asked for the natural map from the symmetric power of a fixed degree of an algebraic variety to the Chow group of zero cycles. It was sketched in Mumford's article \cite{M} and later proved by Roitman  in \cite{R}, that the fiber over zero of a such a natural map is a countable union of Zariski closed subsets of the symmetric power of the given algebraic variety. In this section we are going to adapt the techniques present in the Roitman's proof in \cite{R} to our set up to derive at the fact that $\theta^{-1}([e])$ is a countable union of translates of an algebraic variety.
\begin{proposition}
Let $\theta$ be the natural map from $G$ to $G/\Gamma(e)$ defined as above. Then $\theta^{-1}([e])$ is a countable union of translates of Zariski closed subsets of $G$.
\end{proposition}
\begin{proof}
Consider $\theta^{-1}([e])$. Suppose that $g$ belongs to $\theta^{-1}([e])$, that means that there exists $f:\Gamma\to G$ such that $f(0)=g$ and $f(\infty)=e$. Now the idea in the Roitman's proof is to stratify $\theta^{-1}([e])$ by the degree of $f$. Consider
$$T^d(e)=\{g\in G|\exists f\in \Hom^v(\Gamma,G), f(0)=g,f(\infty)=e\}\;.$$
Here $\Hom^d(\Gamma,G)$ is the hom-scheme parametrizing the degree $d$ morphisms from $\Gamma$ to $G$, it is known to be a quasi-projective subscheme of the Hilbert scheme of $\Gamma\times G$, parametrizing subvarieties of $\Gamma\times G$ having Hilbert polynomial $d$.
It is easy to see that
$$\theta^{-1}([e])=\cup_{d\in\NN} T^d(e)\;.$$
Now we prove that each $T^d(e)$ is a quasi-projective subscheme of $G$. For that consider the Cartesian diagram
$$
  \diagram
  V_d=\Hom^d(\Gamma,G)\times_{G\times G}G\ar[dd]_-{} \ar[rr]^-{} & & G  \ar[dd]^-{} \\ \\
  \Hom^d(\Gamma,G) \ar[rr]^-{ev} & & G\times G
  \enddiagram
  $$
Where the morphism $ev$ is given by
$$ev(f)=(f(0),f(\infty))$$
and the morphism from
$G$ to $G\times G$ is given by $g\mapsto (g,e)$. Then it is easy to check that $T^d(e)$ is nothing but $\pi(V_d)$, where $\pi$ is the projection from $V_d$ to $G$. Since $V_d$ is a quasi-projective scheme, we get that $\pi(V_d)$ is a quasi-projective subscheme of $G$. Therefore $T^d(e)$ is a quasi-projective subscheme of $G$.

Now we prove that $\overline{T^d(e)}$ is a subset of $\theta^{-1}([e])$ proving that $\theta^{-1}([e])$ is a countable union of Zariski closed subsets of $G$.
Let $g$ belongs to $\overline{T^d(e)}$. Then we have to prove that there exists $f:\Gamma\to G$ such that
$$f(0)=g,\quad f(\infty)=e\;.$$
Let $W$ be an irreducible component of $T^d(e)$ whose Zariski closure contains the point $g$. Let $U$ be an affine neighborhood of $g$ such that $U\cap W$ is non-empty. Let us take an irreducible curve $C$ passing through $g$ in $U$. Let $\bar{C}$ be the Zariski closure of $C$ in $\bar{W}$. Now embedding $G$ in $G\times G$ by the homomorphism $g\mapsto (g,e)$, we have the regular morphism
$$ev:\Hom^d(\Gamma,G)\to G\times G$$
given by
$$ev(f)=(f(0),f(\infty))$$
and $T^d(e)$ is the image of the morphism $ev$. Then we can choose a quasi-projective curve $T$ in $\Hom^d(\Gamma,G)$ such that the closure of $ev(T)$ is $\bar{C}$. We give details of the construction of $T$. Let us consider $ev^{-1}(\bar C)$. It is of dimension greater or equal than $1$. So it contains a curve. Consider two distinct points on $C$, consider their inverse images, then there will be a curve in $\ev^{-1}( C)$ which map to the curve $C$. Then this curve is our required curve $T$.

Now let $\bar{T}$ be the closure of $T$ in $\PR^N_k$. Let $\wt{T}$ be the normalization of $\bar{T}$ and let $\wt{T_0}$ be the inverse image of $T$ in $\wt{T}$. Consider the evaluation morphism
$$f_0:\wt{T_0}\times \Gamma\to T\times \Gamma\subset \Hom^d(\Gamma,G)\times \Gamma\stackrel{e}{\to}G$$
where
$$e:(f,t)\mapsto f(t)\;.$$
This defines a rational map from $\wt{T}\times \Gamma$ to $G$, since $\wt{T}$ is non-singular, on each fiber $\wt{T}\times\{Q\}$, $f_0$ defines a regular map from $\wt{T}$ to $\bar{C}$.
 So the regular morphism $\wt{T_0}\to T\to \bar{C}$ extends to a regular morphism $\wt{T}\to \bar{C}$. Let $P$ be a point in the fiber of this morphism over $g$. For any closed $k$-point $Q$ on $\Gamma$, $T\times \{Q\}$ maps onto $\bar C$. Then we get that there exists $x_0,x_1$ on $\Gamma$ such that
 $$f_0|_{\wt{T}\times \{x_0\}}(P)=g,\quad f_0|_{\wt{T}\times \{x_1\}}=e\;.$$
 This gives us that $g$ is $\Gamma$-equivalent to $e$. So we get that $\theta^{-1}([e])$ contains $\overline{T^d(e)}$. So we can write $\theta^{-1}([e])$ as a countable union of Zariski closed subsets of $G$.
\end{proof}
\section{Varieties over uncountable ground fields}
In this section we prove that $\theta^{-1}([e])$ is a countable union of translates of an abelian variety, when the ground field is uncountable.

First we prove the following few lemmas.

\begin{lemma}
\label{lemma1}
Let $X$ be a projective variety over an uncountable ground field $k$. Then $X$ cannot be written as a countable union of proper Zariski closed subsets of itself.
\end{lemma}

\begin{proof}
Suppose that $X$ can be written as a countable union of proper Zariski closed subsets of it. By Noether's normalization there exists a finite map from $X\to \PR^m_k$ where $m=\dim(X)$. Since $X$ can be written as a countable union of Zariski closed subset of itself, we can write $\PR^m_k$ as a countable union of Zariski closed subsets of itself, say
$$\PR^m_k=\cup_{i\in \NN}Z_i\;.$$
Since the collection of $Z_i$'s is countable and the ground field $k$ is uncountable, we get that there exists a hyperplane $H$ not contained in any of the $Z_i$'s. So we can write
$$\PR^{m-1}_k=H=\cup_{i\in \NN}(Z_i\cap H)\;.$$
Continuing this process we obtain that $\PR^1_k$ is a countable union of its $k$-points, which contradicts the assumption that $k$ is uncountable.
\end{proof}
\begin{lemma}
\label{lemma2}
Let $k$ be uncountable. Let $Z=\cup_{i\in \NN}Z_i$ be a countable union of Zariski closed subsets embedded in some $\PR^m_k$. Then we can write $Z$ as a unique irredundant countable union of irreducible Zariski closed subsets of $\PR^m$, that is
$$Z=\cup_{i\in \NN}A_i$$
such that $A_i\not\subset A_j$ for $i\neq j$ and this decomposition is unique.
\end{lemma}
\begin{proof}
We write each $Z_j$ as a finite union of irreducible components
say,
$$Z_i=\cup_{l=1}^{k_i} Z'_{i_l}$$. Then we get that
$$Z=\cup_i (\cup Z'_{i_1}\cup \cdots Z'_{i_{k}})$$
for simplicity we write the above as
$$\cup_i B_i$$
where each $B_i$ is irreducible. Now order the $B_i$'s by set inclusion and only consider those $B_i$'s which are maximal with respect to inclusion. Then we get an irredundant decomposaition
$cup_i B_i\;.$
Now we have to prove that this decomposition is unique. Suppose that there exists another decomposition $\cup_{j\in \NN}A_j$. Then observe that each $A_j$ is contained in some $B_i$, otherwise, we can write
$$A_j=\cup_{i\in \NN} (A_j\cap B_i)$$
where $A_j\cap B_i$ is a proper closed subset of $A_j$, which contradicts the previous lemma \ref{lemma1}. Similarly $B_i$ is contained in some $A_k$, so we get that $A_j=A_k$ and consequently $A_j=B_i$. So we get that the decomposition is unique.
\end{proof}
\begin{proposition}
\label{prop1}
$\theta^{-1}([e])$ is a countable union of translates of an abelian subvariety of $G$.
\end{proposition}
\begin{proof}
By the previous two lemmas \ref{lemma1}, \ref{lemma2} we get that $\theta^{-1}([e])$ is a countable union of Zariski closed closed subsets in $G$ such that the union is irredundant. So let $\theta^{-1}([e])$ is a countable union say $\cup_{i\in \NN}A_i$, such that $A_i\not\subset A_j$ for $i\neq j$. Then we claim that there exists a unique $A_0$ among these $A_i$'s which passes through $e$ and moreover this $A_0$ is an abelian variety.
So suppose that there exists $A_0,\cdots,A_m$ passing through $e$, then consider $A_0+\cdots+A_m$. Since $\theta^{-1}([e])$ is a subgroup of $G$, we get that $A_0+\cdots+A_m$ is a subset of $\theta^{-1}([e])$. Since it is the image of the morphism $A_1\times\cdots\times A_m\to G$ given by
$$(a_1,\cdots,a_m)\mapsto a_1+\cdots+a_m$$
it is irreducible and Zariski closed. Therefore by lemma \ref{lemma1}, it must land inside some $A_j$. Also since $e$ belongs to $A_0,\cdots,A_m$, we get that $A_i\subset A_1+\cdots+A_m\subset A_j$, for all $i=0,\cdots,m$. So by the irredundancy we get that $A_0=A_1=\cdots=A_m$.
So $A_0$ is the unique irreducible Zariski closed subset in the decomposition $\cup_i A_i$ such that it passes through $e$.

Now we claim that $A_0$ is an abelian variety. For that suppose that $x\in A_0$. Then consider $-x+A_0$, since translation by $-x$ is a homeomorphism, $-x+A_0$ is Zariski closed and irreducible. Hence by \ref{lemma1}, it is a subset of some $A_j$. Now $e$ belongs to $-x+A_0$ and there passes a unique $A_0$ through $e$ so we get that $A_j=A_0$ and hence $-x+A_0\subset A_0$. Now we show that $A_0+A_0$ is inside $A_0$. For that we observe that $A_0+A_0$ is the image of the regular morphism from $A_0\times A_0$ to $G$ given by
$$(a,b)\mapsto a+b\;.$$
Then again by lemma \ref{lemma1}, $A_0+A_0$ is inside some $A_j$ and $A_0$ is inside $A_0+A_0$, so we get that $A_j=A_0$. So $A_0$ is an abelian variety.

Therefore we can write
$$\theta^{-1}([e])=\cup_{x\in \theta^{-1}([e])}(x+A_0)$$
where the above union is disjoint. We prove that the above union is actually countable. So let us consider $x+A_0$, since it is Zariski closed irreducible, it must land inside some $A_j$. So we get that $A_0\subset -x+A_j$, by similar argument we get that $-x+A_j\subset A_k$ so we get that $A_k=A_0$,which in turn gives us that $x+A_0=A_j$, since there are only countably many $A_j$'s we get only countably many $x+A_0$'s giving us
$$\theta^{-1}([e])=\cup_{i\in \NN}x_i+A_0\;.$$
\end{proof}
\section{Projective algebraic groups over $\CC$}
In this section we are going to understand that $\theta^{-1}([e])$ is actually a finite union of translates of an abelian subvariety of $G$, when the ground field is complex numbers.
\begin{proposition}
Let $G$ be a projective algebraic group over $\CC$. Then $\theta^{-1}([e])$ is a finite union of translates of $A_0$.
\end{proposition}
Before going to the proof of the theorem we recall the definition of a Baire space. A topological space $X$ is called Baire, if any countable union of closed sets having non-empty interior implies that one of them has non-empty interior. Any complete metric space or a locally compact Hausdorff space is Baire. Also if $X$ is a non-empty Baire space, which is a countable union of closed subsets, then it follows that one of the closed subsets has non-empty interior.
\begin{proof}
By the proposition \ref{prop1}
$$\ker(\theta)=\cup_{i\in \NN}(x_i+A_0)$$
since $A_0$ is an abelian subvariety in $G$ it is closed in the analytic topology. Also $G$ is a metric space and $A_0$ is closed, so it is complete under this metric. Now we claim that $\ker(\theta)$ is complete under this metric. So take a Cauchy sequence $\{y_n\}_n$ in $\ker(\theta)$. We claim that there exists some $n_0\in \NN$ such that for all $n\geq n_0$, $y_n$ belongs to one of the $x_i+A_0$. Suppose the opposite. That is for each $N$, there exists $n,m\geq N$ such  that  $y_n$ belongs to one $x_i+A_0$ and $y_m$ belongs to $x_j+A_0$, where $(x_i+A_0)\cap (x_j+A_0)=\emptyset$. Now given any $\epsilon>0$, there exists $N\in \NN$ such that for $n,m\geq N$ we have
$$d(x_n,x_m)<\epsilon\;.$$
Now take $\epsilon$ to be less than the infimum of $d(y_n,a)$, where $y_n$ belongs to $x_i+A_0$ and $a\in x_j+A_0$, where $(x_i+A_0)\cap (x_j+A_0)=\emptyset$. Then for large $m$ we have
$$d(y_n,y_m)<\epsilon$$
but on the other hand
$$d(y_n,y_m)\geq inf(d(y_n,a))\;,$$
where $a$ varies in $x_j+A_0$.
Since $x_j+A_0$ is compact in the analytic topology we have that, there exists $b$ such that
$$inf(d(y_n,a))=d(y_n,b)\;.$$
Now if $d(y_n,b)=0$ then we have $y_n=b$, but $(x_i+A_0)\cap (x_j+A_0)=\emptyset$. So $d(y_n,b)>0$, therefore choosing $\epsilon$ to be less than $inf (d(y_n,a))$ we get that
$$d(y_n,y_m)\geq \epsilon$$
contradicting the fact that $\{y_n\}_n$ is Cauchy. So there exists $N\in\NN$ such that for all $n\geq N$ we have $y_n$ belongs to one fixed $x_i+A_0$. Since $x+A_0$ is complete for each $x\in G$, we get that the sequence $\{y_n\}_n$ converges in $x_i+A_0$. Hence $\ker(\theta)$ is complete. So it is a Baire space. Therefore there exists one $x_i$ such that the interior of $x_i+A_0$ is non-empty. Since translation by $-x_i$ is a homeomorphism we get that the interior of $A_0$ is non-empty. Now $A_0$ is a topological subgroup of $\ker(\theta)$ whose interior is non-empty. So $A_0$ is open in $\ker(\theta)$. Therefore each $x_i+A_0$ is open in $\ker(\theta)$. So we have an open cover of $\ker(\theta)$. Since $\ker(\theta)$ complete in the given metric, it is closed in $A_0$. So it is compact in the analytic topology of $G$. Therefore we get that a finitely many $x_i+A_0$ cover $\ker(\theta)$. So $\ker(\theta)$ is a finite union of translates of $A_0$. Since each $x_i+A_0$ is Zariski closed and irreducible in $G$, we get that
$\ker(\theta)$ is an algebraic subgroup of $G$.
\end{proof}

\section{Relative rational equivalence and the Mumford-Roitman type argument}

Now we would like to generalize the Mumford-Roitman argument saying that the natural map from the Chow variety of a smooth projective variety to the Chow group of the variety itself, has the fibers equal to a countable union of Zariski closed subsets of the Chow variety. All this is happening over an uncountable ground $k$. Now we suppose that $X$ is a smooth-projective scheme over another Noetherian scheme $S$. Then we observe that there is a natural map from the $k(S)$-points of the Hilbert scheme $\Hilb^d(X/S)$ to $\CH_0(X/S)$. We prove that the fibers of this map is a countable union of Zariski closed subschemes in $\Hilb^d(X/S)(\eta)$, where $\eta$ is the generic point of $S$.

First of all we recall the definition of the relative cycles on the scheme $X/S$ due to \cite{SV}. A relative cycle of relative dimension $r$ on $X$, is an algebraic cycle such that all its prime components maps to the generic point $\eta$ of $S$ and for any $k$-point $P$ on $S$, the pullback with respect to any fat point corresponding to $P$ coincide. Now observe that any $r$-cycle on $X$ which is flat over $S$, that is its composition of $Z\to X\to S$ is flat is a relative cycle. In view of this we consider the Hilbert scheme $\Hilb^d(X/S)$ and its $k(S)$ points, which is nothing but $\Hilb^d(X_{\eta})$, that is the length $d$ zero dimensional subschemes of $X_{\eta}$. We denote it by $X_{\eta}^{[d]}$, and we have a natural map from $X_{\eta}^{[d]}$ to $\bcZ_0(X/S)$ associating a zero dimensional subscheme of length $d$ to its fundamental cycle. For sake of convenience we identify $X_{\eta}^{[d]}$ with its image under the Hilbert-Chow morphism to the symmetric power $ \Sym^d X_{\eta}$, and denote it by the same notation $X_{\eta}^{[d]}$.

We define the rational equivalence on $\bcZ_0(X/S)$ as follows. Let $Z_1,Z_2$ be two relative cycles of relative dimension $0$, they are said to be rationally equivalent if there exists a morphism $f:\PR^1_S\to \Sym^d X$ and a relative effective zero cycle $B$, such that  image of $f$ and support of $B$ is contained in $ X_{\eta}^{[d,d]}$, and $f(0,s)=Z_1+B;, f(\infty,s)=Z_2+B$ In the following we denote $X_{\eta}^{[d_1,\cdots,d_n]}$ to be $\prod_i X_{\eta}^{[d_1]}$.

\begin{proposition}
The natural map $\theta_{X/S}$ from $X_{\eta}^{[d,d]}$ to $\CH_0(X/S)$ has fibers equal to a countable union of Zariski closed subschemes of $X_{\eta}^{[d,d]}$.
\end{proposition}

\begin{proof}
Let $W_d$ be the subset of $X_{\eta}^{[d,d]}$ consisting of pairs $(A,B)$, where $\theta_{X/S}(A,B)$ is relatively rationally equivalent to $0$ on $\CH_0(X/S)$. Let $W_d^{u,v}$ be the subset of $X_{\eta}^{[d,d]}$ which consists of pairs $(A,B)$, such that there exists $f$ in $\Hom^v(\PR^1_S,X_{\eta}^{[d+u,d]})$ with $f(0,s)=(A+C,C)$ and $f(\infty,s)=(B+D,D)$. Then we have $(A,B)$ relatively rationally equivalent. Then it is easy to see that $W_d^{u,v}$ is a subset of $W_d$. On the other hand suppose that $(A,B)$ belongs to $W^d$. Then there exists $f:\PR^1_S\to  X_{\eta}^{[d,d]} $ and a relative zero cycle $C$, such that image of $f$ is contained in $ X_{\eta}^{[d,d]}$ and we have
$$f(0,s)=A+C,f(\infty,s)=B+C$$
then we can find $u,v$ such that $(A,B)$ belongs to $W^d_{u,v}$.

Now we prove that the sets $W_d^{u,v}$ is a quasiprojective variety in $X_{\eta}^{[d,d]}$ and its Zariski closure is contained in $W_d$. Then we can write $W_d$ as a countable union of Zariski closed subsets of $X_{\eta}^{[d,d]}$. Consider the morphism $e$ from $\Hom^v(\PR^1_S,X_{\eta}^{[d+u,u]})\to X_{\eta}^{[d+u,u,d+u,u]}$, by sending a morphism $f$ to the pair $(f_{\eta}(0),f_{\eta}(\infty))$. The other morphism from $X_{\eta}^{[d,u,d,u]}$ to $X_{\eta}^{[d+u,u,d+u,u]}$ given by $(A,C,B,D)\mapsto (A+C,C,B+D,D)$. Then if we consider the fiber product of $\Hom^v(\PR^1_S,X_{\eta}^{[d+u,u]})$ and $X_{\eta}^{[d,u,d,u]}$ over $X_{\eta}^{[d+u,u,d+u,u]}$ and call it $V$. SO $V$ consists of quintuples $(f,A,C,B,D)$ such that
$$f_{\eta}(0)=(A+C,C),\quad f_{\eta}(\infty)=(B+D,D)\;.$$
Then $pr_{2,4}(V)$ is contained in $W^d_{u,v}$. On the other hand suppose that we have an element of $W^d_{u,v}$, the by definition of $W^d_{u,v}$, we have an element of $V$, such that if we apply $pr_{2,4}$ to it we get $(A,B)$. So we get that $pr_{2,4}(V)=W^d_{u,v}$. Since $V$ is a quasiprojective variety its image is quasiprojective.

Now we prove that
$$W_d^{u,v}=pr_{1,2}(\wt{s}^{-1}(W^{0,v}_{d+u}\times W^{0,v}_d))$$
where
$$\wt{s}(A,B,C,D)=(A+C,B+D,C,D)$$
from $X_{\eta}^{[d,d,u,u]}$ to $X_{\eta}^{[d+u,d+u,u,u]}\;.$
Let $(A,B,C,D)$ be such that its image under $\wt{s}$ is in $W^{0,v}_{d+u}\times W^{0,v}_u$. It means that there exists a morphism $g$ in $\Hom^v(\PR^1_S,X_{\eta}^{[d+u]})$ and another $h$ in $\Hom^v(\PR^1_S,X_{\eta}^{[u]})$ such that $g_{\eta}(0)=A+C,g(\infty)=B+D$ and $h_{\eta}(0)=C,h_{\eta}(\infty)=D$. Let us consider $f=g\times h$, then $f$ belong to $\Hom^v(\PR^1_S,X_{\eta}^{[d+u,u]})$, with
$$\theta_{X/S}(f_{\eta}(0))=A,\theta_{X/S}(f(\infty))=B\;.$$
It means that $(A,B)$ belong to $W^d_{u,v}$.

On the other hand suppose that $(A,B)$ belongs to $W^d_{u,v}$. Then there exists $f$ in $\Hom^v(\PR^1_S,X_{\eta}^{[d+u,u]})$ such that
$$f_{\eta}(0)=(A+C,C),f_{\eta}(\infty)=(B+D,D)\;.$$
Compose $f$ with the projections to $X^{[d+u]}$ and to $X^{[u]}$, then we have $g$ in $\Hom^v(\PR^1_S,X_{\eta}^{[d+u]})$ and $h\in\Hom^v(\PR^1_S,X_{\eta}^{[u]})$, such that
$$g_{\eta}(0)=A+C,g_{\eta}(\infty)=B+D$$
and
$$h_{\eta}(0)=C,h_{\eta}(\infty)=D\;.$$
Therefore we have that
$$W_d=pr_{1,2}(\wt{s}^{-1}(W_{d+u}\times W_u))\;.$$
Then we prove that the closure of $W_d^{0,v} $ is contained in $W_d$. Let $(A,B)$ be a closed point in the closure of ${W_d^{0,v}}$. Let $W$ be an irreducible component of ${W_d^{0,v}}$ whose closure contains $(A,B)$. Let $U$ be an affine neighborhood of $(A,B)$ such that $U\cap W$ is non-empty. Then there is an irreducible curve $C$ in $U$ passing through $(A,B)$. Let $\bar{C}$ be the Zariski closure of $C$ in $\bar{W}$. The evaluation map
$$e:\Hom^v(\PR^1_S,X_{\eta}^{[d]})\to X_{\eta}^{[d,d]}$$
given by
$$f\mapsto (f(s,0),f(s,\infty))$$
is regular and $W_d^{0,v}$ is its image. Let us choose a curve $T$ in $\Hom^v(\PR^1_S,X_{\eta}^{[d]})$ such that the closure of $e(T)$ is $\bar C$.  Consider the normalization $\wt{T}$ of the Zariski closure of $T$. Let $\wt{T_0}$ be the pre-image of $T$ in the normalization. Now the regular morphism $\wt{T_0}\to T\to \bar C$ extends to a regular morphism from $\wt{T}$ to $\bar C$. Now let $f$ be a pre-image of $(A,B)$. Then we have $f(0,s)=A;, f(0,\infty)=B$ and the image of $f$ is contained in $X_{\eta}^{[d]}$. Therefore $A,B$ are relatively rationally equivalent.
\end{proof}

\end{document}